\documentclass[12pt,twoside]{amsart}
\usepackage{amssymb,amsmath,amsthm, amscd, enumerate, mathrsfs, upgreek}
\usepackage{graphicx, hhline, tikz}
\usepackage[colorlinks=true,pagebackref,hyperindex]{hyperref}
\usepackage[all]{xy}
\usepackage{color}
\usepackage[backrefs, alphabetic, initials]{amsrefs}
\usepackage{color}  
\usepackage{latexsym}
\usepackage[T1]{fontenc} 
\usepackage{fancyhdr}
\usepackage{enumerate}

\title[Rational curves]
{On the existence of rational curves on projective hyperk\"ahler fourfolds}
\author{Haidong Liu}
\date{\today, version 0.02}
\address{Sun Yat-sen University, Department of mathematics, Guangzhou, 510275, China}
\email{liuhd35@mail.sysu.edu.cn, jiuguiaqi@gmail.com}

\DeclareMathOperator{\NS}{N^1}

\DeclareMathOperator{\pseff}{Pseff}

\DeclareMathOperator{\eff}{Eff}

\DeclareMathOperator{\nef}{Nef}

\newtheorem{thm}{Theorem}[section]
\newtheorem{lem}[thm]{Lemma}

\newtheorem{conj}[thm]{Conjecture}
\newtheorem{cor}[thm]{Corollary}

\theoremstyle{definition}

\newtheorem{rem}[thm]{Remark}
\newtheorem*{ack}{Acknowledgments}

\makeatletter
    
    \@addtoreset{equation}{section}
\makeatother

\begin{document}

\begin{abstract}
We show that if $X$ is a projective hyperk\"ahler fourfold and
there exists a nonzero effective divisor $D$ which is not of bi-elliptic type
and contained in the boundary of the nef cone of $X$,
then  $X$ contains a rational curve.
This is a very special case of Oguiso's conjecture for projective hyperk\"ahler fourfolds.
\end{abstract}

\maketitle 
\tableofcontents

\section{Introduction}\label{sec1}

This paper is devoted to the study of the existence of rational curves 
on projective hyperk\"ahler manifolds, motivated by a conjecture of Oguiso in \cite{oguiso}:

\begin{conj}
[Oguiso's conjecture]
\label{og.conj}
	Let $X$ be a projective hyperk\"ahler or Calabi--Yau manifold.
	Assume that there exists a nonzero Cartier divisor 
	$D$ contained in the boundary of the nef cone of $X$.
	Then, $X$ contains a rational curve.
\end{conj}

Here a simply connected manifold $X$ with trivial canonical bundle 
is called a \textit{hyperk\"ahler manifold} (resp. a \textit{Calabi--Yau manifold})  
if $H^2(X, \mathcal O_X) \cong \mathbb C$ (resp. $H^i(X, \mathcal{O}_X)=0$ for $0<i<\dim X$). 
The Beauville--Bogomolov--Yau decomposition \cite{beauville} shows that 
any projective manifold with trivial canonical bundle can be decomposed, 
up to a finite \'etale cover, 
into the product of abelian manifolds, Calabi--Yau manifolds, and hyperk\"ahler  manifolds. 
While it is well-known that abelian manifolds do not contain rational curves, 
it is expected from different topics that the other two cases do contain rational curves.
From the point of view of birational geometry, it is natural to assume additionally the 
existence of a nonzero nef but not ample Cartier divisor as suggested by Oguiso, which is
closely related to the semiampleness conjecture
(see \cite{kollar}*{Conjecture 51},~\cite{lop1}*{Section 4} 
and ~\cite{lp} for conventions and a survey of the related conjectures;
it is also a version of the so-called Strominger--Yau--Zaslow (SYZ) conjecture 
for projective hyperk\"ahler manifolds, see \cite{verbitsky}*{Conjecture 1.7}).

It is well-known that Conjecture \ref{og.conj} holds true in dimension 2. In dimension 3,
Peternell \cite{peternell} (see also Oguiso \cite{oguiso}) proved  it when the nef divisor $D$ is effective.
When $D$ is (possibly) not effective,
Diverio and Ferretti \cite{df} proved it for $\rho(X)\geq 5$, building on the seminal works of Wilson
~\cites{wilson1,wilson2};
recently, the author and Svaldi \cite{liu-svaldi} proved it for $c_3(X)\neq 0$
by a different approach, 
building on the works \cites{lop, wilson2, wilson3} and the references
therein. Combining these results together, the only remaining case of Oguiso's conjecture in dimension 3
is that $X$ is a Calabi--Yau threefold with $c_3(X)=0$ and $2\leq \rho(X)\leq 4$.
In higher dimensions, very little results are known. In this paper,
we study a very special case 
of Oguiso's conjecture for projective hyperk\"ahler fourfolds.

\begin{thm}[Theorem \ref{thm.main}]
\label{main.thm}
	Let $X$ be a projective hyperk\"ahler fourfold.
	Assume that there exists a nonzero effective divisor 
	$D$ contained in the boundary of the nef cone of $X$.
	Then, $X$ contains a rational curve, except possibly that
	$D$ is of bi-elliptic type.
\end{thm}

In this paper, an effective divisor $D$ on a projective hyperk\"ahler fourfold $X$ 
is said to be of \emph{bi-elliptic type} if there exists a 
 relative minimal model $E$ (see \cite{liu-matsumura}*{Subsection 2.3} for the definition)
of the normalization of $D$ with $f\colon E\to D$ being the corresponding morphism,
such that $E$ admits a fibration $g\colon E\to C$ onto a smooth curve $C$ whose general fiber
is a bi-elliptic surface.

Because there is no abundance theorem
for projective hyperk\"ahler fourfolds,
which is one of the main tools to find rational curves in lower dimensions (see \cite{oguiso, peternell}),
the main idea of our paper turns to find rational curves on $D$.
In this approach, the case of bi-elliptic type seems hard to be excluded.
However, finding elliptic curves on projective hyperk\"ahler manifolds is also important;
so we state the following consequence,
whose proof can be seen in the proof of our main theorem.

\begin{cor}
	Let $X$ be a projective hyperk\"ahler fourfold.
	Assume that there exists a nonzero effective divisor 
	$D$ contained in the boundary of the nef cone of $X$.
	Then, $X$ contains a rational curve or an elliptic curve.
\end{cor}

\begin{ack}
The author is supported by the NSFC (grant no.~12001018).
He would like to thank Professors Chen Jiang, Shin-ichi Matsumura, Roberto Svaldi,
 Qizhen Yin for useful discussions and suggestions.
He would also like to thank Mirko Mauri and Thorsten Beckmann for pointing out some mistakes 
in Theorem \ref{thm.fib} and
that a previous cited \cite{matsushita}*{Lemma 1} is missing in its published version.
\end{ack}

Throughout this paper,  we work over the complex number field $\mathbb C$.
A {\em scheme} is always assumed to be separated and of finite type over $\mathbb{C}$, 
and a {\em variety} is a reduced and irreducible algebraic scheme. 
We will freely use the basic notation 
in \cites{fujino-foundations, huy, kollar-mori, laz}.

\section{Preliminaries}\label{sec2}
\subsection{Cones of divisors}
We adopt the following notation to denote the various cones of divisors on a variety $X$:
\begin{itemize}
    \item $\NS(X)_{\mathbb R}$ denotes the N\'eron--Severi $\mathbb R$-vector space on $X$.

    \item The {\em{nef cone}} $\nef(X)\subset \NS(X)_{\mathbb R}$ is the convex cone of all nef $\mathbb R$-divisor classes on $X$. 
    The interior of $\nef(X)$ is the convex cone of all ample $\mathbb R$-divisor classes. 

    \item The {\em{effective cone}} $\eff(X)\subset \NS(X)_{\mathbb R}$ is the convex cone spanned by
 all effective $\mathbb R$-divisor classes on $X$.

    \item The {\em{pseudoeffective cone}} $\pseff(X)\subset \NS(X)_{\mathbb R}$ is the closure of $\eff(X)$. 
    The interior of $\pseff(X)$, denoted as $\pseff(X)^{\circ}$, is the convex cone of all big $\mathbb R$-divisor classes. 
   
\end{itemize}
It is well-known that $\nef(X)\subset \pseff(X)$ 
(see~\cite{laz}*{Chapter~1}).

\begin{lem}\label{lem.n=p}
Let $X$ be a $\mathbb Q$-factorial terminal projective variety such that $K_X\sim_{\mathbb Q}0$.
If $X$ contains no rational curve, then $\nef(X)=\pseff(X)$. 
\end{lem}

\begin{proof}
For a contradiction, we assume that $\nef(X)\subsetneq \pseff(X)$.
Since both cones are closed, there exists a nonzero effective divisor $D\in \pseff(X)^{\circ}\backslash \nef(X)$.
That is, there exists an effective divisor $D$ which is big but not nef.
Let $\epsilon$ be a sufficiently small number such that $(X, \epsilon D)$ is klt. 
Then, by the cone theorem (see \cite{kollar-mori}*{Theorem 3.7} for 
$\mathbb Q$-divisor case and  \cite{fujino-foundations}*{Theorem 4.5.2} for 
$\mathbb R$-divisor case), 
there exists a rational curve $C$
such that $(K_X+\epsilon D)\cdot C=\epsilon D\cdot C<0$. This is a  contradiction.
\end{proof}

\begin{lem}\label{lem.irr}
Let $X$ be a $\mathbb Q$-factorial terminal projective variety of dimension $n$
such that $K_X\sim_{\mathbb Q}0$. Let $D$ be a nonzero effective divisor 
contained in the boundary of $\nef(X)$.
Assume that $X$ contains no rational curve.
Then, every irreducible component of $D$ is contained in the boundary of $\nef(X)$.
\end{lem}

\begin{proof}
Let $D=\sum_ia_iD_i$ be the irreducible decomposition of $D$, where $a_i>0$ for each $i$.
Since $X$ contains no rational curve, we have that $\nef(X)=\pseff(X)$ by Lemma \ref{lem.n=p}. 
Then, $D$ is also contained in the boundary of $\pseff(X)$, that is, $D$ is nef but not big.
It follows that $D^n=0$.
Moreover, we have that
\[
D_i\in \eff(X)\subset\pseff(X)=\nef(X)
\]
for each $i$.  Therefore,
each term of the  polynomial  expansion on the left of the following equation
\[
(\sum_ia_iD_i)^n=D^n=0
\]
is non-negative. In particular, $D_i^n=0$ for each $i$, that is, $D_i$ is contained in the boundary of $\nef(X)$.
\end{proof}

\subsection{Iitaka dimension and numerical dimension}
In Lemma  \ref{lem.irr}, every irreducible component of $D$
is $\mathbb Q$-Cartier since $X$ is $\mathbb Q$-factorial.
In particular, we can always work in the setting of $\mathbb Q$-structures after replacing $D$ 
with some of its components
by Lemma \ref{lem.irr} for our purpose of this paper.  
Let $D$ be a $\mathbb Q$-Cartier $\mathbb Q$-divisor
on a normal projective variety $X$. Let $m_0$ be a positive integer 
such that $m_0D$ is a Cartier divisor. 
Let $$\Phi_{|mm_0D|}\colon X\dashrightarrow \mathbb P^{\dim\!|mm_0D|}$$ be 
the rational map given by the complete linear system $|mm_0D|$ 
for a positive integer $m$. 
Note that $\Phi_{|mm_0D|}(X)$ denotes the 
closure of the image of the rational map $\Phi_{|mm_0D|}$. 
We set 
\[
\kappa (X, D):=\max_m \dim \Phi_{|mm_0D|}(X)
\] 
if $|mm_0D|\ne \emptyset$ for some $m$ and 
$\kappa (X, D):=-\infty$ otherwise. 
We call $\kappa(X, D)$ ($\kappa(D)$ for short if there is no danger of confusion)
the {\em{Iitaka dimension}} 
of $D$. The \emph{Kodaira dimension} of a reduced and irreducible variety $X$, denoted as $\kappa(X)$,
is defined to be $\kappa(Y,K_Y)$, where $f\colon Y\to X$ is a resolution. The definition
is well-known to be independent of resolutions, and coincides with the 
Iitaka dimension $\kappa(X, K_X)$ when $X$ is normal.

When $D$ is nef, 
we define the {\em{numerical dimension}} $\nu(D)$ of $D$ as
\[
\nu(D):=\max\{h\in \mathbb N \; |  \;D^h\not\equiv 0\}.
\]
It is well-known that $\nu(D) \geq 0$, $\kappa (D)\leq \nu(D) \leq \dim X$ and $\nu(D)=\dim X$ if and only if $D$ is big.

\subsection{Codimension one reduction}\label{sub2.3}
Let $S$ be a smooth variety of dimension $n$ and $D$ be a reduced effective divisor on $S$.
Then, $D$ is simple normal crossing on $S$ in codimension one.
We take a general Kawamata's covering trick $\pi\colon S'\to S$ with respect to $D$
(see \cite{kawamata-cover}*{Lemma 17} or \cite{ev}*{3.19. Lemma}).
Note that in its original setting, $D$ is simple normal crossing on $S$;
however, their methods also work when $D$ is simple normal crossing in codimension one, 
which just leaves a codimension two set mysteriously. 
That is, outside a codimension two set, the covering $\pi\colon S'\to S$ coincides with 
Kawamata's covering trick in the simple normal crossing setting.
In general, $S'$ is normal, but not necessarily smooth.

Let $f\colon X\to S$ be a fibration where $S$ is smooth. 
Then, $f$ is smooth outside a Zariski closed set $Z$ of $S$.
Since there is nothing to do if $\dim Z\leq n-2$, we assume that $Z$
is a reduced divisor. Then, we 
take a general Kawamata's covering trick $\pi\colon S'\to S$  with respect to $Z$ as above, 
and take the normalization $X'$ of the main part of $X\times_{S} S'$.
When the covering $\pi\colon S'\to S$ is taken suitably, 
the fibers of the reduced morphism $f'\colon X'\to S'$ are reduced in codimension one.
Then, such a reduced morphism $f'\colon X'\to S'$ is called a \emph{codimension one reduction of $f$}.
In general, $X'$ and $S'$ might be both very singular; 
but cutting $S$ by general and ample 
hyperplanes $H_1, H_2, \cdots H_{n-1}$, we can restrict the commutative diagram 
to 
\[
\xymatrix{
Y' \ar[d]_-{f'}\ar[r]& Y\ar[d]^-{f}\\
C' \ar[r] & C
}
\]
where $C:=S\cap H_1 \cdots \cap H_{n-1}$ and $C'$ is the pulling back of $C$ by $\pi$.
It is easy to see that this diagram is (an \'etale cover of) a stable reduction.

\section{The main result}\label{sec3}

In this section, we prove the main result of this paper.

\begin{thm}
\label{thm.main}
Let $X$ be a projective hyperk\"ahler fourfold.
	Assume that there exists a nonzero effective divisor 
	$D$ which is not bi-elliptic type
	and contained in the boundary of the nef cone of $X$.
	Then, $X$ contains a rational curve.
\end{thm}

\begin{proof}
By Lemma \ref{lem.irr}, we can replace $D$ with some of its irreducible components
and assume that $D$ is a prime Cartier divisor. Since $X$ is hyperk\"ahler,
it is well-known that the numerical dimension of $D$ is 0,  $\frac{1}{2}\dim X$ or $\dim X$.
Therefore, $\nu(D)=2$ or 4 since  $D$ is nonzero.  
In the case $\nu(D)=4$, that is, $D$ is nef and big, we have that $D$ is semiample
by the basepoint-free theorem.
Since $D$ is not ample, the Iitaka fibration  $\pi\colon X\to Y$ induced by $|kD|$ for $k\gg 0$
is a birational morphism containing non-empty exceptional locus. By \cite{kawamata}*{Theorem 2},
the exceptional locus of $\pi$ is uniruled, hence $X$ contains rational curves.
Therefore, we always assume that $\nu(D)=2$ in the following.

Let us consider the following morphisms
as in the proof of  \cite{liu-matsumura}*{Theorem 3.3}:
\[
 \rho \circ \mu =f \colon E \xrightarrow{\quad \mu \quad}  D'  \xrightarrow{\quad \rho \quad}  D, 
\] 
where $f:= \rho \circ \mu$, $ \rho \colon D'\to D$ is the normalization, 
and $\mu\colon E\to D'$ is a relative minimal model 
(see  \cite{liu-matsumura}*{Subsection 2.3} for the definition).
By subadjunction (see \cite{jiang}*{Proposition 5.1}), 
there exists an effective $\mathbb Q$-divisor $\Delta_{D'}$ on $D'$ such that 
\[
\rho^*D|_{D'}=\rho^*(K_X+D)|_{D'}=K_{D'}+\Delta_{D'}. 
\]
Since $E$ is a relative minimal model,
there exists an effective $\mathbb Q$-divisor $G$ on $E$ such that
$K_E+G=\mu^*K_{D'}$. Set $B:=G+\mu^*\Delta_{D'}$. Then, we have that
\[
K_E+B=\mu^*(K_{D'}+\Delta_{D'})=f^*D|_E.
\]
In particular, $K_E+B$ is a nef Cartier divisor on $E$.
Since $\nu(D)=2$, we also have that $D^2|_D=D^3\equiv 0$. 
It follows that $(K_E+B)^2=(f^*D|_E)^2=D^2|_D=D^3\equiv 0$.
That is, $\nu(K_E+B)\leq 1$.
If $K_E$ is not nef, then choose a rational curve $R$ such that $K_E\cdot R<0$ by the cone theorem.
Since $K_E$ is $\mu$-nef by our construction, $R$ is not contracted by $\mu$.
Then, $f(R)=\rho\circ \mu (R)$ is a rational curve since $\rho$ is finite. 
Therefore, we further assume that $K_E$ is nef. In particular, $\kappa(E)\geq 0$
by the abundance theorem for terminal threefolds.
Since $B$ is effective, we have that 
\[
0\leq\kappa(E)\leq \kappa(K_E+B)\leq \nu(K_E+B)\leq 1.
\]
We discuss case by case according to $\kappa(K_E+B)$ and $\kappa(E)$.

In the case $\kappa(K_E+B)=0$, we have that $\kappa(K_E+B)=\kappa(E)=0$.
Since $K_E$ is nef, we have that $K_E\sim_{\mathbb Q}0$ by the abundance theorem
for terminal threefolds.
If $B\neq 0$, then choose a small rational number $\epsilon $ such that $(E, \epsilon B)$ is klt.
By the abundance theorem for klt pairs of dimension 3, we have that 
$\kappa(\epsilon B)=\kappa(K_E+\epsilon B)=\nu(K_E+\epsilon B)=\nu(\epsilon B)$.
Then, 
\[
\kappa(K_E+B)=\kappa(B)=\nu(B)\geq 1,
\] 
which is a contradiction. Therefore, we assume that $B=0$.
That is, $D$ has at most canonical singularities and $K_E=f^*K_D$. 
However, since $\nu(D)=2$, we have that $D^2\cdot H^2>0$
for some very ample divisor $H$ on $X$.
Then, 
\[
D|_D \cdot H|_D \cdot H|_D=D^2\cdot H^2>0,
\]
which implies that $\nu(K_D)=\nu(D|_D)\geq 1$. This contradicts to $f^*K_D=K_E\sim_{\mathbb Q}0$.
Therefore, this case can never happen.

In the case $\kappa(K_E+B)=1$, we have that $\kappa(K_E+B)= \nu(K_E+B)=1$ and $0\leq \kappa(E)\leq 1$.
If $\kappa(E)=0$, then $K_E\sim_{\mathbb Q}0$ by the abundance theorem as above. 
 It follows that $\kappa(B)= \nu(B)=1$.
By considering the klt pair $(E, \epsilon B)$ and the abundance theorem for klt pairs of dimension 3  as above,
we have that $B$ is semiample. 
Then, there exists a fibration $g\colon E\to C$ onto a smooth curve $C$ induced by $|kB|$ for some $k\gg 0$.
In particular,  $B$ is $g$-vertical. Therefore, $E$ and $D$ are isomorphic outside some fibers of $g$.
Let $F$ be a general fiber of $g$. We still denote $f(F)$ as $F$ if there is no risk of confusion.
Then, $K_F\sim_{\mathbb Q} 0$ by adjunction. In the case where $F$ is a K3 or Enriques surface, 
it is well-known that $F$ contains  rational curves, and so does $D$.
The case where $F$ is a bi-elliptic surface is excluded by our assumption. 
In the final case where $F$ is an abelian surface, 
$C$ is $\mathbb P^1$ or a smooth elliptic curve by the canonical bundle formula.
In the former case, $g$ is of type $\rm{\,I\,}_0$ in the sense of Oguiso \cite{oguiso}*{Main Theorem}
and there exists a rational curve $R$ on $E$ which is $g$-horizontal (see the proof 
of \cite{oguiso}*{Theorem 5.1} or \cite{dfm}*{Corollary 1.7}). In particular, 
$R$ is not contained in $B$ and thus $f(R)$ is a rational curve on $D$.
In the latter case, 
we have an \'etale cover $F\times C'\to E$,
where $C'$ is again a smooth elliptic curve (see \cite{oguiso}*{Appendix}).
Moreover, $f^*D|_E\sim_{\mathbb Q} B\sim_{\mathbb Q} aF$ for some positive rational number $a$. 
It follows that 
\begin{align*}
0&=\int_X D^2(\sigma\overline{\sigma})
=\int_D D|_D(\sigma\overline{\sigma})=\int_E f^*D|_E(f^*\sigma f^*\overline{\sigma}) \\
&=a\int_E F(f^*\sigma f^*\overline{\sigma})=a\int_F f^*\sigma f^*\overline{\sigma}
=a\int_F \sigma\overline{\sigma},
\end{align*}
where the first equation follows from $q_X(D, D)=0$ by Fujiki's lemma.
Therefore, $F$ is a holomorphic Lagrangian subsurface which is a complex torus.
Then, there exists an almost holomorphic Lagrangian fibration $h\colon X\dasharrow  S$ 
from a solution of Beauville's question (see \cites{cop, glr, hw}).
Hence, $X$ contains a rational curve  by the following Theorem \ref{thm.fib}.

Finally, if $\kappa(E)=1$, then $1=\kappa(E)\leq \kappa(K_E+\epsilon B)\leq \kappa(K_E+B)=1$ for any 
rational number $0\leq \epsilon \leq 1$. In particular, we can choose $\epsilon$ sufficiently small
such that $(E, \epsilon B)$ is klt and $ \kappa(K_E+\epsilon B)=1$. 
If $K_E+\epsilon B$ is not nef, then there exists 
a rational curve $R$ such that $(K_E+\epsilon B)\cdot R<0$ by the cone theorem.
Since $K_E$ is $\mu$-nef, $R$ is not contracted by $\mu$.
Then $f(R)$ is a rational curve on $D$ as above. Therefore, we assume that
 $K_E+\epsilon B$ is nef and
thus semiample by the abundance theorem. 
Let $g\colon E\to C$ be the Iitaka fibration onto a smooth curve $C$ 
induced by $|k(K_E+\epsilon B)|$ for some $k\gg 0$.
Let $B_F$ be an effective $\mathbb Q$-divisor on $F$
such that $K_F+B_F=(K_E+\epsilon B)|_F\sim_{\mathbb Q} 0$. 
If $B$ is  $g$-horizontal, then $B_F$ is  nonzero.
It follows that $F$ is a uniruled surface.
Therefore, $E$ and $D$ are  both covered by rational curves.
If $B$ is  $g$-vertical, then $B_F=0$ and $K_F\sim_{\mathbb Q} 0$.
Therefore, as above, we only need to consider the case where $F$ is an abelian surface.
In this case, the fibration $g$ induced by $|k(K_E+\epsilon B)|$ 
coincides with that induced by $|kK_E|$ for $k\gg 0$.
In particular, $K_E\equiv bF$ and 
$ B\sim_{\mathbb Q}(K_E+\epsilon B)-K_E\equiv c F$ for two positive rational numbers $b, c$.
It follows that 
\begin{align*}
0&=\int_X D^2(\sigma\overline{\sigma})
=\int_D D|_D(\sigma\overline{\sigma})=\int_E f^*D|_E(f^*\sigma f^*\overline{\sigma}) \\
&=\int_E (K_E+B)(f^*\sigma f^*\overline{\sigma})=(b+c)\int_E F(f^*\sigma f^*\overline{\sigma})
=(b+c)\int_F \sigma\overline{\sigma}.
\end{align*}
Therefore,  $F$ is a holomorphic Lagrangian subsurface which is a complex torus.
Then,  as above, there exists an almost holomorphic Lagrangian fibration $h\colon X\dasharrow  S$ 
from a solution of Beauville's question and
$X$ contains a rational curve by the following Theorem \ref{thm.fib}.
\end{proof}

\begin{rem}
Considering Oguiso's conjecture for Calabi--Yau fourfolds, the case
$\nu(D)=3$ is easy to deal with as in \cite{liu-matsumura}*{Remark 4.2}; the case  $\nu(D)=1$
can be reduced to the case $\nu(D)=3$ by Lemma \ref{lem.n=p} and taking a general line passing through
$D$ in $\mathbb P(H^2(X,\mathbb R))$ as in \cites{dfm, liu-svaldi,oguiso, wilson1} and the references therein.
In the remaining case  $\nu(D)=2$, parts of the proof of Theorem \ref{thm.main} also works except the case 
where a general fiber of (a relative minimal model of) $D$ is abelian or bi-elliptic.
\end{rem}

Part of the proof of Theorem \ref{thm.main} heavily depends on the prominent results of \cites{cop, glr, hw},
where an \emph{almost holomorphic Lagrangian fibration} is defined to be a dominant meromorphic map
$h\colon X\dasharrow  S$ such that there is a Zariski open
subset $U\subset S$ and the restriction $h|_{h^{-1}(U)}\colon h^{-1}(U)\to U$ is holomorphic and proper
with connected fibers, and the reduction of every irreducible component of a fiber of $h|_{h^{-1}(U)}$ is a Lagrangian subvariety of $X$. In the rest of this section, 
we prove the final puzzle of Theorem \ref{thm.main}.

\begin{thm}\label{thm.fib}
Let $h\colon X\dasharrow S$  be an almost holomorphic Lagrangian fibration 
from a  projective hyperk\"ahler fourfold $X$ onto a normal surface $S$.
Then, $X$ contains a rational curve.
\end{thm}

\begin{proof}
By \cite{glr}*{Theorem 6.2}, there exists a commutative diagram:
\[
\xymatrix{
X \ar@{-->}[d]_-{h}\ar@{-->}[r]& \widetilde X\ar[d]^-{\widetilde h}\\
S \ar@{-->}[r] & \widetilde S
}
\]
where $\widetilde h$ is a holomorphic Lagrangian fibration 
on a projective hyperk\"ahler manifold $\widetilde X$ (a birational model of $X$)
and the horizontal maps are birational. If the rational map $ \rho\colon \widetilde X\dasharrow X$ is not 
everywhere defined, then $X$ contains rational curves by \cite{kollar-mori}*{Corollary 1.5}.
Therefore, we assume that $\rho\colon \widetilde X\to X$ is a morphism.
Since $\widetilde X$ and $X$ have the same Hodge diamonds (see \cite{huy}*{Corollary 4.7}),
$\rho$ has to be an identity (otherwise, the non-trivial fiber of $\rho$ will contribute
difference to the Hodge structures of $\widetilde X$ and $X$ in some weight corresponding to
the dimension of the non-trivial fiber). Therefore, we can always assume that 
$h\colon X\to S$  is a holomorphic Lagrangian fibration after replacing $h$ by $\widetilde h$
in the sequel.

For a contradiction, we further assume that $X$ contains no rational curve.
In particular, $X$ contains no ruled surface.
Then, by \cite{oguiso}*{Appendix, Theorem B.1}, every fiber of $h$ is a smooth abelian 
or bi-elliptic surface. Let $C$ be the union of all points on $S$ over which
the fibers are bi-elliptic surfaces.
Let 
\[
\xymatrix{
X' \ar[d]_-{h'}\ar[r]^-{\varphi}& X\ar[d]^-{h}\\
S'\ar[r]_-\pi & S.
}
\]
be a codimension one reduction with respect to $C$ as in Subsection \ref{sub2.3}.
Note that $\pi\colon S' \to S$ is not only ramified over $C$.
Let $C'=\pi^{-1}(C)$. Since $\pi$ is a codimension one reduction,
the multiplicity of the fiber over the generic point of every component of $C'$ is $1$. 
Then, any fiber of $h'$ is a smooth abelian surface 
outside a  subset $P'$ of codimension 2 (see the proof of \cite{oguiso}*{Appendix, Theorem B.1}).
Set $P:=\pi(P')$.
Let $H$ be a general smooth curve on $S$ avoiding the set $P$.
Set $Y_H=h^{-1}(H)$, $H'=\pi^{-1}(H)$, and $Y'_{H'}=h'^{-1}(H')$.
Since $H$ is general, we can assume that $H'$ and $Y'_{H'}$ are smooth.
 Then, there exists an induced commutative diagram
as follows:
\[
\xymatrix{
Y'_{H'} \ar[d]_-{h'}\ar[r]^-{\varphi}& Y_H\ar[d]^-{h}\\
H'\ar[r]_-\pi & H,
}
\]
where $h'$ is smooth and $\varphi$ is an \'etale cover  by \cite{oguiso}*{Appendix, Theorem B.1}.
By moving $H$ (thus $H'$), we can see that 
 $\varphi\colon X'\to X$ is a quasi-\'etale cover (finite cover \'etale in codimension one),
possibly ramified over the singular fibers $h^{-1}(P)$.
Since the fundamental group of  $X\backslash h^{-1}(P)$ is trivial, 
$\varphi$ has to be an identity. It follows that $\pi$ has to be an identity.
Since $H$ is closed and the classifying space $\mathcal D$ of abelian surfaces is open,
the period map $J\colon H\to \mathcal D$ is a constant.
By moving $H$, we can see that
the period map $J\colon S_0\to \mathcal D$ 
is a constant, where $S_0:=S\backslash P$.
Then, $J$ extends to a constant morphism $J\colon S\to \mathcal D$. 
That is, $h$ is an isotrivial morphism.
However, since $S=\mathbb P^2$ by \cites{hx, ou}, $S$ is simply connected;
it follows that the period map $J\colon S\to \mathcal D$ factors through
the period domain, which is the Siegel upper half-plane $\mathbb H_2$.
In particular, $X=F\times S$, where $F$ is the general fiber of $h$.
It follows that $q(X)=h^1(F\times S, \mathcal O_{F\times S})\geq 1$, which is a contradiction.
\end{proof}

\begin{rem}
Let $h\colon X\to S$  be a holomorphic Lagrangian fibration 
from a  projective hyperk\"ahler fourfold $X$ onto a normal surface $S$.
We actually showed  in Theorem \ref{thm.fib}
that the subset of $S$ over which the fibers are ruled
surfaces can not be of codimension 2.
Indeed, if this is not the case, then with the same notation and arguments, 
we have that $\varphi\colon X'=F\times S'\to X$ is a 
quasi-\'etale cover, which is impossible as above.

Therefore, there must exists a curve $C$ on $S$
such that $h^{-1}(C)=\sum Z_i$, where $Z_i$ is a fibration of ruled surfaces with $\dim Z_i=3$.
Note that a rational curve deforms in a family of dimension at least $2$
by \cite{ran} (see also \cite{av1}*{Theorem 4.1}).
Then, every $Z_i$ is just an
irreducible component of the locus covered by the deformation of 
some rational curve contained in some ruled surface.
\end{rem}

\begin{rem}
If $X$ is a  hyperk\"ahler  fourfold with algebraic dimension $0<a(X)<4$,
then the algebraic reduction map
induces a holomorphic Lagrangian fibration (see \cite{cop}*{Theorem 1.2} for example).
In this case, very similar arguments as in Theorem \ref{thm.fib} show that $X$ contains rational curves.
On the other hand, a general  hyperk\"ahler  fourfold contains neither effective divisors nor curves
(see \cite{huy}*{(1.17)}), thus contains no rational curves.
\end{rem}


The semiampleness (SYZ) conjecture predicts that a projective hyperk\"ahler manifold $X$ 
admits a birational contraction or a holomorphic Lagrangian fibration 
induced by the nonzero nef divisor $D$ contained in the boundary of $\nef(X)$.
In this case, as showed in Theorems \ref{thm.main} and \ref{thm.fib}, it is reasonable to predict that 
there exists a rational curve $C$ contained in some singular fiber
such that $D\cdot C=0$. By these observations, we propose the following conjecture,
which is a parallel version of \cite{av1}*{Theorem 1.9} (and the references therein)
and a generation of \cite{liu-matsumura}*{Conjecture 1.1} for hyperk\"ahler manifolds.

\begin{conj}
Let X be a projective hyperk\"ahler manifold. 
Let $D$ be a nef Cartier divisor on $X$ such that $D\cdot C>0$ for any rational curve $C$ on $X$.
Then, $D$ is ample.
\end{conj}


\end{document}